\documentclass[twoside,11pt]{amsart}
\usepackage{amsmath,amssymb,amsthm}

\begin{document}

\setlength{\textwidth}{145mm} \setlength{\textheight}{203mm}


\frenchspacing
\newcommand{\ie}{i.e. }
\newcommand{\X}{\mathfrak{X}}
\newcommand{\W}{\mathcal{W}}
\newcommand{\F}{\mathcal{F}}
\newcommand{\T}{\mathcal{T}}
\newcommand{\U}{\mathcal{U}}
\newcommand{\M}{(M,\f,\xi,\eta,g)}
\newcommand{\Lf}{(G,\f,\xi,\eta,g)}
\newcommand{\R}{\mathbb{R}}
\newcommand{\s}{\mathfrak{S}}
\newcommand{\n}{\nabla}
\newcommand{\f}{\varphi}
\newcommand{\D}{{\rm d}}
\newcommand{\al}{\alpha}
\newcommand{\bt}{\beta}
\newcommand{\gm}{\gamma}
\newcommand{\lm}{\lambda}
\newcommand{\ta}{\theta}
\newcommand{\om}{\omega}
\newcommand{\ea}{\varepsilon_\alpha}
\newcommand{\eb}{\varepsilon_\beta}
\newcommand{\eg}{\varepsilon_\gamma}
\newcommand{\sx}{\mathop{\mathfrak{S}}\limits_{x,y,z}}
\newcommand{\norm}[1]{\left\Vert#1\right\Vert ^2}
\newcommand{\nf}{\norm{\n \f}}
\newcommand{\Span}{\mathrm{span}}

\newcommand{\thmref}[1]{The\-o\-rem~\ref{#1}}
\newcommand{\propref}[1]{Pro\-po\-si\-ti\-on~\ref{#1}}
\newcommand{\secref}[1]{\S\ref{#1}}
\newcommand{\lemref}[1]{Lem\-ma~\ref{#1}}
\newcommand{\dfnref}[1]{De\-fi\-ni\-ti\-on~\ref{#1}}
\newcommand{\corref}[1]{Corollary~\ref{#1}}



\numberwithin{equation}{section}
\newtheorem{thm}{Theorem}[section]
\newtheorem{lem}[thm]{Lemma}
\newtheorem{prop}[thm]{Proposition}
\newtheorem{cor}[thm]{Corollary}

\theoremstyle{definition}
\newtheorem{defn}{Definition}[section]

\hyphenation{Her-mi-ti-an ma-ni-fold ah-ler-ian}


\title[A natural connection on some classes]
{A natural connection on some classes of almost contact manifolds
with B-metric}

\author{Mancho Manev, Miroslava Ivanova}

\maketitle


{\small
\textbf{Abstract} 

Almost contact manifolds with B-metric are considered. A special linear
connection is introduced, which preserves the almost contact B-metric structure on these
manifolds. This connection is investigated on some classes of the considered manifolds.


} %

\section{Introduction}

In this work we consider an almost contact B-metric manifold
denoted by $\M$, i.e. a $(2n+1)$-dimensional differentiable
manifold $M$ equipped with an almost contact structure
$(\f,\xi,\eta)$ and a B-metric $g$. The B-metric $g$ is a
pseudo-Riemannian metric of signature $(n,n+1)$ with the opposite
compatibility of $g$ and $(\f,\xi,\eta)$ by comparison with the
compatibility for the known almost contact metric structure.
Moreover, the B-metric is an odd-dimensional analogue of the
Norden (or anti-Hermitian) metric on almost complex manifolds.

Recently, manifolds with neutral metrics and various tensor
structures have been an object of interest in mathematical
physics.

The geometry of  almost contact B-metric manifolds is the geometry
of the structures  $g$ and $(\f,\xi,\eta)$. The linear
connections, with respect to which $g$ and $(\f,\xi,\eta)$ are
parallel, play an important role in this geometry. The structures
$g$ and $(\f,\xi,\eta)$ are parallel with respect to the
Levi-Civita connection $\nabla$ of $g$ if and only if $\M$ belongs
to the class $\F_0:\n\f=0$. Therefore, outside of the class
$\F_0$, the Levi-Civita connection $\n$ is no longer a connection
with respect to which $g$ and $(\f,\xi,\eta)$ are parallel. In the
general case, on $\M$ there exist a countless number of linear
connections with respect to which these structures are parallel.
They are the so-called natural connections.

In this paper we introduce a natural connection on $\M$ which we call a $\f$B-connection.
It is an odd-dimensional analogue of
the known B-connection on almost complex manifolds with Norden metric introduced in \cite{GanGriMih2}.

The $\f$B-connection is studied in
\cite{ManGri2}, \cite{Man3} and \cite{Man4} on the classes $\F_1$, $\F_4$, $\F_5$,
$\F_{11}$. These are the classes, where $\n\f$ is expressed explicitly by the structures $g$ and $(\f,\xi,\eta)$.
In this paper we consider the $\f$B-connection on other classes of
almost contact B-metric manifolds.

The paper is organized as follows. In Sec.~2 we furnish some
necessary facts about the considered manifolds.
In Sec.~3 we define the $\f$B-connection. On some classes of almost contact B-metric manifolds
we characterize the
torsion tensor and the curvature tensor  of this connection.
In Sec.~4, on a 5-dimensional Lie group considered as an almost
contact B-metric manifold in a basic class, we establish that the
$\f$B-connection is flat and it has a parallel torsion.

\section{Almost contact manifolds with B-metric}

Let $(M,\f,\xi,\eta,g)$ be an almost contact manifold with
B-metric or an \emph{almost contact B-metric manifold}, i.e. $M$
is a $(2n+1)$-dimensional differentiable manifold with an almost
contact structure $(\f,\xi,\eta)$ consisting of an endomorphism
$\f$ of the tangent bundle, a vector field $\xi$, its dual 1-form
$\eta$ as well as $M$ is equipped with a pseudo-Riemannian metric
$g$  of signature $(n,n+1)$, such that the following algebraic
relations are satisfied
\begin{equation*}\label{str}
\begin{array}{c}
\f\xi = 0,\quad \f^2 = -Id + \eta \otimes \xi,\quad
\eta\circ\f=0,\quad \eta(\xi)=1,\\
g(\f x, \f y ) = - g(x, y ) + \eta(x)\eta(y)
\end{array}
\end{equation*}
for arbitrary $x$, $y$ of the algebra $\X(M)$ on the smooth vector
fields on $M$.
The structural group of $\M$ is $G\times I$, where $I$ is the
identity on $\Span(\xi)$ and $G=\mathcal{GL}(n;\mathbb{C})\cap
\mathcal{O}(n,n)$.

Further, $x$, $y$, $z$, $w$ will stand for arbitrary elements of
$\X(M)$ or vectors in the tangent space $T_pM$, $p\in M$.

In \cite{GaMiGr}, a classification of the almost contact manifolds with B-metric is made with respect
to the (0,3)-tensor $F(x,y,z)=g\bigl(
\left( \nabla_x \f \right)y,z\bigr)$, where $\nabla$ is the
Levi-Civita connection of $g$. The tensor $F$ has the properties
\begin{equation*}\label{1.3}
F(x,y,z)=F(x,z,y)=F(x,\f y,\f
z)+\eta(y)F(x,\xi,z)+\eta(z)F(x,y,\xi).
\end{equation*}

The following relations are valid (\cite{Man}):
\begin{equation}\label{eta-xi}
    F(x,\f
    y,\xi)=\left(\n_x\eta\right)y=g\left(\n_x\xi,y\right),\qquad
    \eta\left(\n_x \xi\right)=0.
\end{equation}

The basic
1-forms associated with $F$ are:
\begin{equation*}\label{lee}
    \theta(z)=g^{ij}F(e_i,e_j,z),\quad
    \theta^*(z)=g^{ij}F(e_i,\f e_j,z),\quad
    \omega(z)=F(\xi,\xi,z).
\end{equation*}
Furthermore, $g^{ij}$ are the components of the inverse matrix of
$g$ with respect to a basis $\{e_i,\xi\}$ of $T_pM$.

The basic classes in the classification from \cite{GaMiGr} are $\F_1$,
$\F_2$, $\dots$, $\F_{11}$. Their intersection is the class $\F_0$
determined by the condition $F=0$. It is clear that $\F_0$ is the
class of almost contact B-metric manifolds with $\n$-parallel
basic structures, i.e. $\n\f=\n\xi=\n\eta=\n g=0$.

In \cite{Man}, it is proved that the class
$\U=\F_4\oplus\F_5\oplus\F_6\oplus\F_7\oplus\F_8\oplus\F_9$ is
defined by the conditions
\begin{equation}\label{F4-9}
    F(x,y,z)=\eta(y)F(x,z,\xi)+\eta(z)F(x,y,\xi),\qquad
    F(\xi,y,z)=0.
\end{equation}

Further, we consider the subclasses $\U_1$, $\U_2$ and $\U_3$ of $\U$, defined as follows:
\begin{enumerate}
  \item[a)] The subclass
$\U_1=\F_4\oplus\F_5\oplus\F_6\oplus\F_9\subset\U$
is determined by \eqref{F4-9} and the condition
$\D\eta=0$;
  \item[b)] The subclass
$\U_2=\F_4\oplus\F_5\oplus\F_6\oplus\F_7\subset\U$
is determined by \eqref{F4-9} and the condition
\begin{equation}\label{F4-7}
F(x,y,\xi)=-F(\f x,\f y,\xi);
\end{equation}
  \item[c)] The subclass
$\U_3=\F_4\oplus\F_5\oplus\F_6\subset\U$
is determined by \eqref{F4-9} and the conditions
\begin{equation*}\label{F4-6}
F(x,y,\xi)=F(y,x,\xi)=-F(\f x,\f y,\xi).
\end{equation*}
\end{enumerate}

The classes $\F_4$, $\F_5$ and $\F_6$ are determined in $\U_1$ by
the conditions $\theta^*=0$, $\theta=0$ and $\theta=\theta^*=0$,
respectively.


\section{The $\f$B-connection}


\begin{defn}[\cite{ManGri2}]
A linear connection $\n'$ is called a \emph{natural connection} on
  $(M,\f,\xi,\eta,g)$ if the almost contact structure
$(\f,\xi,\eta)$ and the B-metric $g$ are parallel with respect to
$\n'$, i.e. $\n'\f=\n'\xi=\n'\eta=\n'g=0$.
\end{defn}
\begin{prop}[\cite{ManIv36}]\label{prop-nat}
A linear connection $\n'$ is natural on $\M$ if and only if
$\n'\f=\n'g=0$.
\end{prop}

A natural connection  exists on
any almost contact manifold with B-metric and coincides with the
Levi-Civita connection only on a $\F_0$-manifold.

If $\n$ is the Levi-Civita connection, generated by $g$, then we
denote
\begin{equation}\label{1}
\n'_xy=\n_xy+Q(x,y).
\end{equation}
Furthermore, we use the denotation
$Q(x,y,z)=g\left(Q(x,y),z\right)$.
\begin{prop}[\cite{Man31}]\label{prop-KT}
The linear connection $\n'$, determined by \eqref{1}, is a natural
connection on a mani\-fold
$(M,\f,\xi,\eta,g)$ if and only if %
\begin{gather}
 Q(x,y,\f z)-Q(x,\f y,z)=F(x,y,z),\label{1a}\\
 Q(x,y,z)=-Q(x,z,y).\label{1b}
\end{gather}
\end{prop}

As an odd-dimensional analogue of
the B-connection on almost complex manifolds with Norden metric, introduced in \cite{GanGriMih2},
we give the following
\begin{defn}
A linear connection $\n'$, determined by
\begin{equation}\label{fB}
\n'_xy=\n_xy+\frac{1}{2}\bigl\{\left(\n_x\f\right)\f
y+\left(\n_x\eta\right)y\cdot\xi\bigr\}-\eta(y)\n_x\xi
\end{equation}
on an almost contact manifold with B-metric $\M$, is called a \emph{$\f$B-connection}.
\end{defn}
According to \propref{prop-nat}, we establish that the
$\f$B-connection is a natural connection.

Let us remark that the connection $\n'$ determined by \eqref{fB} is studied in
\cite{ManGri2}, \cite{Man3} and \cite{Man4} on manifolds from the classes $\F_1$, $\F_4$, $\F_5$,
$\F_{11}$. In this paper we consider the $\f$B-connection on manifolds from $\U$ and its subclasses $\U_1$, $\U_2$, $\U_3$.

Further $\n'$ will stand for the $\f$B-connection.

\subsection{Torsion properties of the $\f$B-connection}

Let $T$ be the torsion of $\n'$, i.e. $T(x,y)=\n'_x y-\n'_y
x-[x, y]$ and the corresponding (0,3)-tensor $T$
determined by $T(x,y,z)=g(T(x,y),z)$.

The connection $\n'$ on $\M$ has a torsion tensor and the torsion forms as
follows:
\begin{equation*}\label{TD}
\begin{split}
T(x,y,z)=&-\frac{1}{2}\bigl\{F(x,\f y,\f^2 z)-F(y,\f
x,\f^2 z)\bigr\}+\eta(x)F(y,\f z,\xi)\\
&-\eta(y)F(x,\f z,\xi)+\eta(z)\bigl\{F(x,\f y,\xi)-F(y,\f
x,\xi)\bigr\},
\end{split}
\end{equation*}
\begin{gather}
t(x)=\frac{1}{2}\left\{\ta^*(x)+\ta^*(\xi)\eta(x)\right\},\quad
    t^*(x)=-\frac{1}{2}\left\{\ta(x)+\ta(\xi)\eta(x)\right\},\nonumber\\
    \hat{t}(x)=-\om(\f x),\nonumber
\end{gather}
where the torsion forms are defined by
\[
t(x)=g^{ij}T(x,e_i,e_j),\quad t^*(x)=g^{ij}T(x,e_i,\f e_j),\quad
\hat{t}(x)=T(x,\xi,\xi).
\]

In \cite{ManIv36}, we have obtained the following decomposition in
11 factors $\T_{ij}$ of the vector space
$
\T=\left\{T(x,y,z)\; \vert\;\;
T(x,y,z)=-T(y,x,z)\right\}
$
of the torsion (0,3)-tensors on $\M$. This decomposition is
orthogonal and invariant with respect to the structural group
$G\times I$.

\begin{thm}[\cite{ManIv36}]\label{thm-can}
The torsion $T$ of $\n'$ on $\M$ belongs to the
 subclass $\T_{12}\oplus\T_{13}\oplus\T_{14}\oplus\T_{21}\oplus\T_{22}\oplus\T_{31}
\oplus\T_{32}\oplus\T_{33}\oplus\T_{34}\oplus\T_{41}$ of $\T$,
where
\[
\begin{split}
\T_{12}:\quad &T(\xi,y,z)=T(x,y,\xi)=0,\\
              &T(x,y,z)=-T(\f x,\f y,z)=T(\f x,y,\f z);\\
\T_{13}:\quad &T(\xi,y,z)=T(x,y,\xi)=0,\\
              &T(x,y,z)-T(\f x,\f y,z)=\sx T(x,y,z)=0;\\
\T_{14}:\quad &T(\xi,y,z)=T(x,y,\xi)=0,\\
              &T(x,y,z)-T(\f x,\f y,z)=\sx T(\f x,y,z)=0;\\
\T_{21}:\quad &T(x,y,z)=\eta(z)T(\f^2 x,\f^2 y,\xi),\quad
                T(x,y,\xi)=-T(\f x,\f y,\xi); \\
\T_{22}:\quad &T(x,y,z)=\eta(z)T(\f^2 x,\f^2 y,\xi),\quad
                T(x,y,\xi)=T(\f x,\f y,\xi);\\
\T_{31}:\quad &T(x,y,z)=\eta(x)T(\xi,\f^2 y,\f^2 z)-\eta(y)T(\xi,\f^2 x,\f^2 z),\\
                &T(\xi,y,z)=T(\xi,z,y)=-T(\xi,\f y,\f z); \\
\T_{32}:\quad &T(x,y,z)=\eta(x)T(\xi,\f^2 y,\f^2 z)-\eta(y)T(\xi,\f^2 x,\f^2 z),\\
                &T(\xi,y,z)=-T(\xi,z,y)=-T(\xi,\f y,\f z); \\
\T_{33}:\quad &T(x,y,z)=\eta(x)T(\xi,\f^2 y,\f^2 z)-\eta(y)T(\xi,\f^2 x,\f^2 z),\\
                &T(\xi,y,z)=T(\xi,z,y)=T(\xi,\f y,\f z); \\
\end{split}
\]
\[
\begin{split}%
\T_{34}:\quad &T(x,y,z)=\eta(x)T(\xi,\f^2 y,\f^2 z)-\eta(y)T(\xi,\f^2 x,\f^2 z),\\
                &T(\xi,y,z)=-T(\xi,z,y)=T(\xi,\f y,\f z);\\
\T_{41}:\quad
&T(x,y,z)=\eta(z)\left\{\eta(y)\hat{t}(x)-\eta(x)\hat{t}(y)\right\}.
\end{split}
\]
\end{thm}

Bearing in mind \eqref{F4-9}, the $\f$B-connection on a manifold
in $\U$ has the form
\begin{equation}\label{fB6}
    \n'_xy=\n_xy+\left(\n_x\eta\right)y.\xi-\eta(y)\n_x\xi
\end{equation}
and for its torsion and
torsion forms are valid:
\begin{gather}
T(x,y)=\D\eta(x,y)\xi+\eta(x)\n_y\xi-\eta(y)\n_x\xi,\nonumber
\\
    T(x,y,z)=\eta(x)F(y,\f z,\xi)-\eta(y)F(x,\f
    z,\xi)+\eta(z)\D\eta(x,y),\label{T4-9}\\
t(x)=\theta^*(\xi)\eta(x),\quad t^*(x)=-\theta(\xi)\eta(x),\quad \hat{t}(x)=0.\nonumber
\end{gather}
Thus we establish that $t$ and $t^*$ coincide with $\theta^*$ and
$-\theta$, respectively, on the considered manifolds.

Bearing in mind \eqref{T4-9}, we obtain $T(\xi,y,z)=F(y,\f z,\xi)$
in $\U$ and thus
\begin{equation}\label{Txi}
\begin{split}
    T(x,y,z)&=\eta(x)T(\xi,y,z)-\eta(y)T(\xi,x,z)\\
    &+\eta(z)T(\xi,x,y)-\eta(z)T(\xi,y,x).
\end{split}
\end{equation}
Using \thmref{thm-can}, we establish
\begin{prop}\label{prop-T-4-9}
The torsion $T$ of $\n'$ on $\M\in\U$ belongs to the class
$\T_{21}\oplus\T_{22}\oplus\T_{31}\oplus\T_{32}\oplus\T_{33}\oplus\T_{34}$.
Moreover, if $\M$ belongs to $\U_1$, $\U_2$ and $\U_3$,
then $T$ belongs to $\T_{31}\oplus\T_{33}$, $\T_{31}\oplus\T_{32}$
and $\T_{31}$, respectively.
\end{prop}

The 1-form $\eta$ is closed on the manifolds in $\U_1$ and then
\eqref{T4-9} implies
\begin{equation}\label{T4569}
    T(x,y,z)=\eta(x)F(y,\f z,\xi)-\eta(y)F(x,\f
    z,\xi).
\end{equation}
According to \eqref{1} and \eqref{fB6}, we have in $\U$ the
expression
\[
Q(x,y,z)=F(x,\f y,\xi)\eta(z)-F(x,\f z,\xi)\eta(y).
\]

The latter two equalities imply the following
\begin{prop}
The torsion $T$ of $\n'$ on $(M,\f,\xi,\eta,g)\in\U_1$ has the properties
$
Q(x,y,z)=T(z,y,x)$ and $\sx T(x,y,z)=0.
$
\end{prop}

Since $g(x,\f y)=g(\f x,y)$ holds and \eqref{F4-7} is valid on a
manifold from $\U_2$, then we obtain the following property using
\eqref{eta-xi}
\begin{equation}\label{fxi}
    \n_{\f x}\xi=\f\n_x\xi.
\end{equation}

\begin{prop}
The torsion $T$ of $\n'$ on $(M,\f,\xi,\eta,g)\in\U_3$ has the following properties:
\begin{gather}
T(\f x,y,z)+T(x,\f y,z)-T(x,y,\f z)=0;\label{csumf}\\
T(x,y,z)+T(\f x,y,\f z)+T(x,\f y,\f z)=0;\label{csumf2}\\
T(x,\f y,\f z)=T(x,\f z,\f y);\label{csumf3}\\
T(x,y,z)-T(x,z,y)=\eta(y)T(x,\xi,z)-\eta(z)T(x,\xi,y).\label{csumf4}
\end{gather}
\end{prop}
\begin{proof}
Property \eqref{csumf} for $\U_3=\U_1\cap\U_2$ follows from
\eqref{T4569} for $\U_1$ and \eqref{fxi} for $\U_2$.
By virtue of \eqref{T4569} we obtain
\begin{gather}
T(x,\f y,\f z)=\eta(x)F(\f y,\f^2 z,\xi)=-\eta(x)F(y,\f
z,\xi),\label{TF}\\
T(\f x,y,\f z)=-\eta(y)F(\f x,\f^2 z,\xi)=\eta(y)F(x,\f
z,\xi),\nonumber
\end{gather}
bearing in mind \eqref{F4-7} for the class $\U_2$. Summing up the
equalities in the last two lines and giving an account of
\eqref{T4569} again, we obtain \eqref{csumf2}. Because of
\eqref{TF}, \eqref{eta-xi} and $\D\eta=0$, we prove
\eqref{csumf3}. Equality \eqref{csumf3} and $T(x,y,\xi)=0$ imply
property \eqref{csumf4}.
\end{proof}

\subsection{Curvature properties of the $\f$B-connnection}

Let $R=\left[\n,\n\right]-\n_{[\ ,\ ]}$ be the curvature
(1,3)-tensor of $\nabla$. We denote the curvature $(0,4)$-tensor
by the same letter: $R(x,y,z,w)$ $=g(R(x,y)z,w)$.
The scalar curvature $\tau$ for $R$ as
well as its associated quantity $\tau^*$ are defined respectively by
$\tau=g^{ij}g^{kl}R(e_k,e_i,e_j,e_l)$ and $\tau^*=g^{ij}g^{kl}R(e_k,e_i,e_j,\f e_l)$.
Similarly, the curvature tensor $R'$, the scalar curvature $\tau'$  and its associated quantity $\tau'^*$ for $\n'$ are defined.

According to \eqref{1} and \eqref{1b}, the curvature tensor $R'$
of a natural connection $\n'$ has the following form \cite{KoNo}
\begin{equation*}\label{R'RQ}
\begin{array}{l}
    R'(x,y,z,w)=R(x,y,z,w)+\left(\nabla_x Q\right)(y,z,w)-\left(\nabla_y
    Q\right)(x,z,w)
  \\
\phantom{R'(x,y,z,w)=}
       +g\left(Q(x,z),Q(y,w)\right)-g\left(Q(y,z),Q(x,w)\right).
\end{array}
\end{equation*}%
Then, using \eqref{fB6}, we obtain  the following
\begin{prop}
For the curvature tensor $R'$ and the scalar curvature $\tau'$ of $\n'$ on $\M\in\U$
we have
\begin{gather}
\begin{split}\label{R'RU}
    R'(x,y,z,w)=R(x,y,\f^2z,\f^2w)
       &-(\n_x\eta)z (\n_y\eta)w \\
       &+(\n_y\eta)z (\n_x\eta)w,
\end{split}\\
    \tau'=\tau-2\rho(\xi,\xi)-\norm{\nabla \xi},\label{t'tU}
\end{gather}
where $\norm{\nabla \xi}=g^{ij}
    g\bigl(\nabla_{e_i} \xi,\nabla_{e_j}
    \xi\bigr)$.
\end{prop}

A curvature-like tensor $L$, i.e. a tensor with properties
\begin{gather}
    L(x,y,z,w)=-L(y,x,z,w)=-L(x,y,w,z),\label{1.10}\\
    \sx L(x,y,z,w)=0, \label{1.10B}
\end{gather}
 we call a \emph{$\f$-K\"ahler-type tensor}
on $\M$ if the following identity is valid \cite{ManGri2}
\begin{equation}\label{1.12}
    L(x,y,\f z,\f w)=-L(x,y,z,w).
\end{equation}
The latter property is characteristic for $R$ on an
$\F_0$-manifold.

It is easy to verify that $R'$ of the $\f$B-connection satisfies
\eqref{1.10} and \eqref{1.12} but it is not a $\f$-K\"ahler-type
tensor because of the lack of \eqref{1.10B}. Because of
\eqref{R'RU}, the condition $\sx R'(x,y)z=0$ holds if and only if
$\sx\{ \eta(x)R(y,z)\xi\}=\sx\{\D\eta(x,y)\n_z\xi\}$ is valid.
Hence it follows the equality $\sx\{\eta(x)R(y,z)\xi\}=0$ for the
class $\U_1$, which implies $R(\f x,\f y)\xi=0$. Then, using
\eqref{fxi} for $\U_2$, we obtain in $\U_3=\U_1\cap\U_2$ the
following form of $R$
\begin{equation}\label{Rxi}
R(x,y,z,\xi)=\eta(x)g\left(\n_y\xi,\n_z\xi\right)-\eta(y)g\left(\n_x\xi,\n_z\xi\right).
\end{equation}

Vice versa, \eqref{Rxi} implies
$\sx\{\eta(x)R(y,z)\xi\}=0$ and then we have
\begin{thm}
The curvature tensor of $\n'$ is a
$\f$-K\"ahler-type tensor on $\M\in\U_1$ if and only if $\M\in\U_3$.
\end{thm}

\section{A Lie group as a 5-dimensional $\F_6$-manifold}

Let us consider the example given in \cite{Man33}. Let $G$ be a
5-dimensional real connected Lie group and let $\mathfrak{g}$ be
its Lie algebra. Let $\left\{e_i\right\}$ be a
global basis of left-invariant vector fields of $G$.
\begin{thm}[\cite{Man33}]\label{thm-G}
Let $(G,\f,\xi,\eta,g)$ be the almost contact B-metric manifold,
determined by
\begin{equation*}\label{f}
\begin{array}{c}
\f e_1 = e_3,\quad \f e_2 = e_4,\quad \f e_3 =-e_1,\quad \f e_4 =
-e_2,\quad \f e_5 =0;\\
\xi=e_5;\qquad \eta(e_i)=0\; (i=1,2,3,4),\quad \eta(e_5)=1;
\end{array}
\end{equation*}
\begin{equation*}\label{g}
\begin{array}{c}
g(e_1,e_1)=g(e_2,e_2)=-g(e_3,e_3)=-g(e_4,e_4)=g(e_5,e_5)=1,
\\
g(e_i,e_j)=0,\; i\neq j,\quad  i,j\in\{1,2,3,4,5\};
\end{array}
\end{equation*}
\[
\begin{array}{l}
\left[e_1,\xi\right]=\lm_1e_1+\lm_2e_2+\lm_3e_3+\lm_4e_4,\\
\left[e_2,\xi\right]=\mu_1e_1-\lm_1e_2+\mu_3e_3-\lm_3e_4,\\
\left[e_3,\xi\right]=-\lm_3e_1-\lm_4e_2+\lm_1e_3+\lm_2e_4,\\
\left[e_4,\xi\right]=-\mu_3e_1+\lm_3e_2+\mu_1e_3-\lm_1e_4.
\end{array}
\]
Then $(G,\f,\xi,\eta,g)$
belongs to the class $\F_6$.
\end{thm}

Further,  $(G,\f,\xi,\eta,g)$ will stand for the $\F_6$-manifold determined by the conditions of \thmref{thm-G}.

Using \eqref{fB6}, we compute the
components of the $\f$B-connection $\n'$ on $\Lf$.  The non-zero components of $\n'$ are the following:
\begin{equation}\label{F6n'}
\begin{array}{l}
\n'_{\xi}e_1=-\frac{1}{2}\left(\lm_2-\mu_1\right)e_2-\frac{1}{2}\left(\lm_4-\mu_3\right)e_4, \\
\n'_{\xi}e_2=\frac{1}{2}\left(\lm_2-\mu_1\right)e_1+\frac{1}{2}\left(\lm_4-\mu_3\right)e_3,\\
\n'_{\xi}e_3=\frac{1}{2}\left(\lm_4-\mu_3\right)e_2-\frac{1}{2}\left(\lm_2-\mu_1\right)e_4,\\
\n'_{\xi}e_4=-\frac{1}{2}\left(\lm_4-\mu_3\right)e_1+\frac{1}{2}\left(\lm_2-\mu_1\right)e_3.
\end{array}
\end{equation}
By virtue of the latter equalities we establish that the basic components
of the curvature tensor $R'$ of $\n'$ are zero and then we have
\begin{prop}
The manifold $\Lf$ has a flat $\f$B-connection.
\end{prop}

Bearing in mind \eqref{Txi}, the torsion of
$\n'$  is determined only by the components $T_{5ij}=T(\xi,e_i,e_j)$. We compute them using $T(\xi,e_i,e_j)=\left(\n_{e_i}\eta\right)e_j$,
$T(\xi,e_i)=\n_{e_i}\xi$ and the equalities (\cite{Man33})
\[
\begin{array}{l}
\n_{e_1}\xi=\lm_1e_1+\frac{1}{2}\left(\lm_2+\mu_1\right)e_2+\lm_3e_3+\frac{1}{2}\left(\lm_4+\mu_3\right)e_4, \\
\n_{e_2}\xi=\frac{1}{2}\left(\lm_2+\mu_1\right)e_1-\lm_1e_2+\frac{1}{2}\left(\lm_4+\mu_3\right)e_3-\lm_3e_4, \\
\n_{e_3}\xi=-\lm_3e_1-\frac{1}{2}\left(\lm_4+\mu_3\right)e_2+\lm_1e_3+\frac{1}{2}\left(\lm_2+\mu_1\right)e_4, \\
\n_{e_4}\xi=-\frac{1}{2}\left(\lm_4+\mu_3\right)e_1+\lm_3e_2+\frac{1}{2}\left(\lm_2+\mu_1\right)e_3-\lm_1e_4.
\end{array}
\]
Thus we get that the non-zero components are:
\begin{equation}\label{Tijk}
\begin{array}{l}
T_{511}=-T_{522}=-T_{533}=T_{544}=\lm_1,\\
T_{512}=T_{521}=-T_{534}=-T_{543}=\frac{1}{2}\left(\lm_2+\mu_1\right),\\
T_{513}=-T_{524}=T_{531}=-T_{542}=-\lm_3,\\
T_{514}=T_{523}=T_{532}=T_{541}=-\frac{1}{2}\left(\lm_4+\mu_3\right).
\end{array}
\end{equation}

Using \eqref{F6n'} and \eqref{Tijk}, we obtain the following
\begin{prop}
The $\f$B-connection $\n'$  on $\Lf$ has a parallel torsion $T$ with respect to $\n'$.
\end{prop}

\bigskip

\small{
\noindent
\textsl{Mancho Manev, Miroslava Ivanova\\
Department of Geometry\\
Faculty of Mathematics and Informatics\\
University of Plovdiv\\
236 Bulgaria Blvd\\
4003 Plovdiv, Bulgaria}
\\
\texttt{e-mails: mmanev@uni-plovdiv.bg, mirkaiv@uni-plovdiv.bg}
}

\end{document}